\documentclass[a4paper,12pt]{article}
\usepackage{times, url}
\usepackage{graphicx,amsmath,amssymb,amsthm,amsfonts,bm,float,cite}
\newtheorem{theorem}{Theorem}[section]

\newtheorem{lemma}{Lemma}[section]

\newtheorem{proposition}{Proposition}[section]

\usepackage[none]{hyphenat}

\usepackage[center]{caption}
\newcommand{\IP}[1]{\textbf{E}\left( #1 \right)}
\newcommand{\FP}[1]{\left\{ #1 \right\}}

\newcommand{\setU}[1]{\mathbb{U}\left( #1 \right)}

\newcommand{\doublepariii}[1]{\left(\!\!\!\left(#1\right)\!\!\!\right)}
\newcommand{\doubleparii}[1]{\left(\!\!\left(#1\right)\!\!\right)}
\newcommand{\doublepari}[1]{\left(\!\left(#1\right)\!\right)}

\usepackage[bottom,symbol]{footmisc}
\DeclareFontEncoding{TS1}{}{}
\DeclareFontSubstitution{TS1}{cmr}{m}{n}
\DeclareTextSymbol{\tcrp}{TS1}{'251}
\DeclareTextSymbolDefault{\tcrp}{TS1}

\begin{document}
\setcounter{page}{1}

\thispagestyle{empty} 


\begin{center}
{\LARGE \bf A proof of Spence's formula using the \\[4mm] reciprocity law for Dedekind sums} 
\vspace{8mm}

{\Large \bf Steven Brown$^1$}
\vspace{3mm}

$^1$ 48 rue Pottier, 78150 Le Chesnay Rocquencourt \\
e-mail: \url{steven.brown.math@gmail.com}
\end{center}
\vspace{9mm}  

{\bf Abstract:} In 1963, Edward Spence published a proof of the following \\

\textit{With $\phi$ being Euler totient function, if $n>1$ is an integer, and if 
\begin{equation*}
0<a_1<\cdots<a_{\phi(n)}<n,
\end{equation*}
are the positive integers less than $n$, coprime with $n$, then
\begin{equation*}
\sum_{j=1}^{\phi(n)}ja_j = \frac{\phi(n)}{24}\left(8n\phi(n)+6n+2\phi(m)(-1)^{\omega(m)}-2^{\omega(m)}\right),
\end{equation*}
where $m$ is the square-free part of $n$ and $\omega(m)$ is the number of prime factors of $m$.} \\

Spence's proof relies on an ingenious observation considering Nagell's totient function. 
Later in 1971, Lucien Van Hamme provided an alternative proof of the result using Fourier analysis and previous work from Hubert Delange in 1968. In this paper I propose another proof of the formula using the reciprocity law for Dedekind sums. If the formula is of interest on its own, it also plays a role in the analysis of the distribution of the $a_j$ as suggested by the work from Hubert Delange. \\

{\bf Keywords:} Spence formula, Dedekind sums, Arithmetical functions \\ 
{\bf 2020 Mathematics Subject Classification:} 11A99, 11F20, 11A25
\vspace{5mm}

\section{Introduction}

\subsection{The claim of interest}

\begin{theorem} Let $n>1$ be an integer, and $m$ is its square-free part. Let's write $a_i$ for $i$ ranging from 1 to $\phi(n)$ the positive integers, less than $n$ and coprime with $n$, in ascending order, then
\begin{equation}\label{eq:Spence_formula}
\sum_{j=1}^{\phi(n)}ja_j= \frac{\phi(n)}{24}\left(8n\phi(n)+6n+2\phi(m){(-1)}^{\omega(m)}-2^{\omega(m)} \right).
\end{equation}
\end{theorem}

This formula has first been found and proven by Spence in 1963 in \cite{spence1963formulae}. It was subsequently proved by Lucien Van Hamme in 1971 in \cite{van1971generalisation}. In this paper we propose another proof of the formula using the reciprocity law for Dedekind sums.

\section{Notations and definitions}

For any real number $x$, we note $\IP{x}$ its integer part and $\FP{x}$ its fractional part. Unless specified otherwise, the letter $n$ denotes an integer greater than 1, and the letter $m$ denotes its square-free part. We note $\phi(n)$ Euler totient function, $\omega(n)$ the function that counts the number of distinct prime factors of $n$, and $\mu(n)$ Möbius function. Also we denote by $\mathbb{P}$ the set of prime numbers. The set of the $\phi(n)$ positive integers coprime with $n$ and less than $n$ is noted

\begin{equation}
\setU{n} := \left\{ a_1, a_2,\ldots, a_{\phi(n)} \right\}.
\label{eq:def_set_U_n}
\end{equation}

For positive integers $a$ and $b$ we use the notation $s(b,a)$ as in Rademacher's book \cite{rademacher1972dedekind} to denote the classical Dedekind sum\footnote{In his book, Rademacher assumes $a$ and $b$ to be relatively prime. In the context of this paper, it is convenient to remove this condition and this poses no issue for $s(a,b)$ to be evaluated. The coprimality condition will be checked whenever required, for example in the case of the application of Dedekind's reciprocity law.}

\begin{equation}
s(b,a):=\sum_{k=1}^a \doublepariii{\frac{kb}{a}}\doublepariii{\frac{k}{a}}.
\label{eq:def_Dedekind_sum}
\end{equation}

with the symbol $\doublepari{x}$ defined by

\begin{equation}\label{eq:def_dp_x}
\doublepari{x} := \left\{
    \begin{array}{ll}
        \FP{x}-\frac{1}{2} & \mbox{if  } x\notin\mathbb{Z} \\
				0                 & \mbox{if  } x\in\mathbb{Z}
    \end{array}
\right. .
\end{equation}

Let's introduce the functions

\begin{equation}
\theta_n(x):=\sum_{d\mid n}\mu(d)\IP{\frac{x}{d}},
\label{eq:def_theta_n_x}
\end{equation}

and

\begin{equation}
\nu_n(x):=\sum_{d\mid n}\mu(d)\FP{\frac{x}{d}}.
\label{eq:def_nu_n_x}
\end{equation}

\section{Basic properties}

We provide below some basic propositions that can be skipped or referred to as and when needed.

\begin{proposition} For any integer $n\geq 1$, and for any real number $x$,
\begin{equation}\label{eq:theta_n_x_plus_nu_n_x}
\theta_n(x)+\nu_n(x) = x\frac{\phi(n)}{n}.
\end{equation}
\end{proposition}

\begin{proof} For any real number $x$ and for any divisor $d$ of $n$ there is $\frac{x}{d}=\IP{\frac{x}{d}} + \FP{\frac{x}{d}}$. The formula appears after multiplication by $\mu(d)$, sommation over all divisors $d$ of $n$ and application of Theorem 2.3 page 26 in \cite{apostol1976introduction}.
\end{proof}

\begin{proposition}\label{prop:transformation_sum_Un_any_f} For any positve integer $n$, and for any function $f$ defined on positive integers,
\begin{equation}
\sum_{a \in \setU{n}} f(a) = \sum_{d\mid n} \mu(d) \sum_{k=1}^{n/d} f(dk).
\label{eq:transformation_sum_Un_any_f}
\end{equation}
\end{proposition}

\begin{proof} Using Theorem 2.1 page 25 in \cite{apostol1976introduction} we have for any integer $k$
\begin{equation*}
\mathbf{1}_{\gcd(k,n)=1}=\sum_{d\mid \gcd(k,n)}\mu(d),
\end{equation*}
hence
\begin{equation*}
\sum_{a \in \setU{n}} f(a) = \sum_{k=1}^n f(k) \sum_{d\mid \gcd(k,n)}\mu(d).
\end{equation*}
Now, we wish to invert the sums. If $d$ is a fixed divisor of $n$ then $d$ divides $\gcd(k,n)$ if and only if $k$ is a multiple of $d$ which provides equation \ref{eq:transformation_sum_Un_any_f}.
\end{proof}

\begin{lemma}\label{lemma:sum_dpx_plus_dpmx_equals_zero} For any $x \in \mathbb{R}$,
\begin{equation}
\doublepari{x}+\doublepari{-x}=0.
\end{equation}
\end{lemma}

\begin{proof} If $x$ is an integer, so is $-x$ and the identity is obviously satisfied. 
If $x$ is not an integer, then $\FP{x}+\FP{-x}=-\IP{x}-\IP{-x}$ is an integer which must equal 1 as it lies in $]0,2[$, then
\begin{equation*}
\doublepari{x}+\doublepari{-x}=\FP{x}-\frac{1}{2}+\FP{-x}-\frac{1}{2}=0.
\end{equation*}
\end{proof}

\begin{proposition}\label{prop:sum_a_in_Un_dp_a_over_d_equal_zero} Let $n>1$ be an integer, and $d$ a divisor of $n$,
\begin{equation}
\sum_{a \in \setU{n}} \doubleparii{\frac{a}{d}} = 0.
\end{equation}
\end{proposition}

\begin{proof} Let $S=\sum_{a \in \setU{n}} \doublepari{\frac{a}{d}}$. Since $a\mapsto n-a$ is a bijection of $\setU{n}$, we easily get to
\begin{equation*}
2S=\sum_{a \in \setU{n}} \left( \doubleparii{\frac{a}{d}} + \doubleparii{-\frac{a}{d}} \right).
\end{equation*}
Considering lemma \ref{lemma:sum_dpx_plus_dpmx_equals_zero}, we conclude $S=0$ as the summands are all equal to 0.
\end{proof}

\begin{proposition}\label{prop:sum_j1bm1_dpjaoverb_equal_zero} Let $a$ and $b$ be positive integers,
\begin{equation}
\sum_{j=1}^{b-1} \doubleparii{j\frac{a}{b}} = 0.
\end{equation}
\end{proposition}

\begin{proof} Let $S=\sum_{j=1}^{b-1} \doubleparii{j\frac{a}{b}} = 0$. By considering the sum in reverse order with $j\mapsto b-j$, we get to
\begin{equation*}
2S=\sum_{j=1}^{b-1} \left( \doubleparii{j\frac{a}{b}} + \doubleparii{-j\frac{a}{b}} \right).
\end{equation*}
We conclude $S=0$ from lemma \ref{lemma:sum_dpx_plus_dpmx_equals_zero}, as in proposition \ref{prop:sum_a_in_Un_dp_a_over_d_equal_zero}
\end{proof}

\begin{proposition}\label{prop:sab_equals_sacbc} For any positive integers $a$, $b$, and $c$,
\begin{equation}
s(ac,bc)=s(a,b).
\end{equation}
\end{proposition}

\begin{proof} We have
\begin{equation*}
s(ac,bc) = \sum_{k=1}^{bc} \doublepariii{\frac{ka}{b}}\doublepariii{\frac{k}{bc}}.
\end{equation*}
The expression $\doubleparii{\frac{ka}{b}}\doubleparii{\frac{k}{bc}}$ is equal to 0 when $k=bc$ and when $k=0$, therefore we can take the sum from 0 to $bc-1$. We may as well write $k=bi+j$ for $0\leq i<c$, and for $0\leq j<b$, so that
\begin{equation*}
s(ac,bc) = \sum_{i=0}^{c-1}\sum_{j=0}^{b-1} \doublepariii{\frac{ja}{b}}\doublepariii{\frac{bi+j}{bc}}.
\end{equation*}
Considering that the summand is 0 when $j=0$, we can take the second sum from 1 to $b-1$. In that case $0<\frac{bi+j}{bc}<1$, and we can see that
\begin{equation*}
\doublepariii{\frac{bi+j}{bc}} = \left\{\frac{bi+j}{bc}\right\}-\frac{1}{2} = \frac{i}{c}+\frac{j}{bc}-\frac{1}{2}.
\end{equation*}
Considering lemma \ref{prop:sum_j1bm1_dpjaoverb_equal_zero}, we get to
\begin{equation*}
s(ac,bc) = \sum_{i=0}^{c-1}\sum_{j=1}^{b-1} \doublepariii{\frac{ja}{b}}\frac{j}{bc} = \sum_{j=1}^{b-1} \doublepariii{\frac{ja}{b}}\frac{j}{b}.
\end{equation*}
Considering lemma \ref{prop:sum_j1bm1_dpjaoverb_equal_zero} again to add $-\frac{1}{2}$, and given that $\frac{j}{b}=\left\{\frac{j}{b}\right\}$ when $0<j<b$, this is also
\begin{equation*}
s(ac,bc) = \sum_{j=1}^{b-1} \doublepariii{\frac{ja}{b}}\left(\left\{\frac{j}{b}\right\}-\frac{1}{2}\right) = \sum_{j=1}^{b-1} \doublepariii{\frac{ja}{b}}\doublepariii{\frac{j}{b}} = s(a,b).
\end{equation*}
\end{proof}

\begin{proposition}\label{prop:formule_fn_Delange} Let $n>1$ be an integer,
\begin{equation}
\sum_{d_1 \mid n} \sum_{d_2 \mid n} \mu(d_1) \mu(d_2) \frac{d_1d_2}{n^2}\gcd\left(\frac{n}{d_1},\frac{n}{d_2}\right)^2 = \frac{2^{\omega(n)}\phi(n)}{n}.
\label{eq:formule_fn_Delange}
\end{equation}
\end{proposition}

\begin{proof} The formula was proven\footnote{As the reference is in French and not easily accessible I also provide the proof here.} by Hubert Delange in \cite{delange1968distribution} pages 82 and 83. Let $f(n)$ be the left hand side of equation \ref{eq:formule_fn_Delange}. If $d_1$ and $d_2$ are two divisors of $n$, then
\begin{equation}
d_1 d_2 \gcd\left(\frac{n}{d_1},\frac{n}{d_2}\right) = n \gcd(d_1,d_2).
\label{eq:an_gcd_identity}
\end{equation}
This can be seen directly by counting the valuation of any prime $p$ on both sides of the equation. Let's note $A$ the valuation of $p$ on the left hand side of equation \ref{eq:an_gcd_identity},
\begin{align*}
A
&= v_p(d_1)+v_p(d_2)+\min\left(v_p(n)-v_p(d_1),v_p(n)-v_p(d_2)\right), \\
& =v_p(d_1)+v_p(d_2)+v_p(n)-\max(v_p(d_1),v_p(d_2)).
\end{align*}
If $B$ is the valuation of $p$ on the right hand side of equation \ref{eq:an_gcd_identity},
\begin{equation*}
B = v_p(n)+\min(v_p(d_1),v_p(d_2)),
\end{equation*}
Thus, for any prime $p$,
\begin{equation*}
B-A=\min(v_p(d_1),v_p(d_2))+\max(v_p(d_1),v_p(d_2))-v_p(d_1)-v_p(d_2)=0,
\end{equation*}
And we conclude that equation \ref{eq:an_gcd_identity} is valid. We now take the square of equation \ref{eq:an_gcd_identity} and reformulate $f(n)$ as follows
\begin{equation*}
f(n) =  \sum_{d_1 \mid n} \sum_{d_2 \mid n} \mu(d_1) \mu(d_2) \frac{\gcd(d_1,d_2)^2}{d_1d_2}.
\end{equation*}
Let's show that the function $f$ is multiplicative. If $a$ and $b$ are two integers relatively prime, then the divisors of $ab$ are simply all the combinations of the divisors of $a$ times the divisors of $b$,
\begin{align*}
f(ab)
&= \sum_{d_1 \mid ab} \sum_{d_2 \mid ab} \mu(d_1) \mu(d_2) \frac{\gcd(d_1,d_2)^2}{d_1d_2}, \\
&= \sum_{\substack{d_1 \mid a\\ \delta_1 \mid b}} \sum_{\substack{d_2 \mid a\\ \delta_2 \mid b}} \mu(d_1\delta_1) \mu(d_2\delta_2) \frac{\gcd(d_1\delta_1,d_2\delta_2)^2}{d_1\delta_1d_2\delta_2} ,\\
&= \left( \sum_{d_1 \mid a} \sum_{d_2 \mid a} \mu(d_1) \mu(d_2) \frac{\gcd(d_1,d_2)^2}{d_1d_2} \right) \left( \sum_{\delta_1 \mid b} \sum_{\delta_2 \mid b} \mu(\delta_1) \mu(\delta_2) \frac{\gcd(\delta_1,\delta_2)^2}{\delta_1\delta_2} \right), \\
&= f(a)f(b).
\end{align*}
Now to evaluate $f$ we only need to evaluate $f(p^\alpha)$ for any prime number $p$, and for any positive integer exponent $\alpha$,
\begin{align*}
f(p^\alpha)
&= \mu(1)^2 + 2\frac{\mu(1)\mu(p)}{p} + \mu(p)^2, \\
&= 2\left(1-\frac{1}{p}\right).
\end{align*}
Finally,
\begin{equation*}
f(n)=\prod_{\substack{p\mid n\\ p\in \mathbb{P}}}2\left(1-\frac{1}{p}\right)=\frac{2^{\omega(n)}\phi(n)}{n}.
\end{equation*}
\end{proof}

\section{A direct proof of Spence formula}

Let $n>1$ be an integer, and $m$ its square-free part. The function $\theta_n(x)$ defined in equation \ref{eq:def_theta_n_x} counts the number of positive integers coprime with $n$ and less than or equal to $x$ (see chapter 2, exercise 9 page 47 in \cite{apostol1976introduction}). We can therefore write\footnote{Note that this rewriting allows for the ascending order condition on the $a_i$ to be removed.}

\begin{equation}\label{eq:transform_jaj_thetanaa}
\sum_{j=1}^{\phi(n)}ja_j = \sum_{a\in \setU{n}}\theta_n(a)a.
\end{equation}

The consideration of equation \ref{eq:theta_n_x_plus_nu_n_x} implies that

\begin{equation}
\sum_{a\in \setU{n}}\theta_n(a) = \frac{\phi(n)}{n}\sum_{a\in \setU{n}}a^2  -\sum_{a\in \setU{n}}\nu_n(a)a. \label{eq:main}
\end{equation}

The formula for the sum of squares of numbers coprime with $n$ is well known and documented, see for example exercise 2.16 in \cite{apostol1976introduction}. One can also derive it directly from proposition \ref{prop:transformation_sum_Un_any_f}

\begin{equation}
\sum_{a\in \setU{n}}a^2 = \frac{\phi(n)}{6}\left( 2 n^2 + m (-1)^{\omega(m)} \right).
\label{eq:formula_sum_ainUn_a2}
\end{equation}

Therefore we only need to focus on the second term. For any $d \mid n$ with $d>1$, we have $\frac{a}{d}\notin \mathbb{Z}$, since $d$ is coprime with $a$ (because $n$ is coprime with $a$). In that case $\left\{\frac{a}{d}\right\}=\doublepari{\frac{a}{d}}+\frac{1}{2}$. When $d=1$, $\left\{\frac{a}{1}\right\}=\doublepari{\frac{a}{1}}=0$. 
Hence

\begin{equation*}
\nu_n(a) = \sum_{d\mid n}\mu(d)\left\{\frac{a}{d}\right\} = -\frac{1}{2} + \sum_{d\mid n}\mu(d)\left(\doubleparii{\frac{a}{d}}+\frac{1}{2}\right).
\end{equation*}

Then we have

\begin{equation*}
\sum_{a\in \setU{n}}\nu_n(a)a = -\frac{1}{2}\frac{n\phi(n)}{2}+\sum_{d\mid n}\mu(d)\sum_{a\in \setU{n}} a\doubleparii{\frac{a}{d}}.
\end{equation*}

Lets observe that $a\in \setU{n}$ may be rewritten as\footnote{In that case $\FP{\frac{a}{n}}=\frac{a}{n}$, and $\frac{a}{n}\notin \mathbb{Z}$}

\begin{equation*}
a = n\left(\doubleparii{\frac{a}{n}}+\frac{1}{2}\right).
\end{equation*}

After substitution and simplification by means of proposition \ref{prop:sum_a_in_Un_dp_a_over_d_equal_zero}, we get to

\begin{equation*}
\sum_{a\in \setU{n}}\nu_n(a)a = -\frac{n\phi(n)}{4}+n\sum_{d\mid n}\mu(d)\sum_{a\in \setU{n}} \doubleparii{\frac{a}{n}}\doubleparii{\frac{a}{d}}.
\end{equation*}

Now, using proposition \ref{prop:transformation_sum_Un_any_f} with $f(x)=\doublepari{\frac{x}{n}}\doublepari{\frac{x}{d}}$, and changing the name of the shadow variable $d$ to $d_1$, leads to

\begin{equation*}
\sum_{a\in \setU{n}}\nu_n(a)a = -\frac{n\phi(n)}{4} + n \sum_{d_1\mid n}\mu(d_1) \sum_{d_2\mid n}\mu(d_2)\sum_{k=1}^{n/d_2} \doubleparii{\frac{kd_2}{n}}\doubleparii{\frac{kd_2}{d_1}}.
\end{equation*}

The sum over $k$ can be expressed as the Dedekind sum $s\left(\frac{n}{d_1},\frac{n}{d_2}\right)$ and

\begin{equation}
\sum_{a\in \setU{n}}\nu_n(a)a = -\frac{n\phi(n)}{4} + n \sum_{d_1\mid n}\sum_{d_2\mid n} \mu(d_1)\mu(d_2) s\left(\frac{n}{d_1},\frac{n}{d_2}\right). \label{eq:expr_sum_ainUn_nunaa}
\end{equation}

Let's define

\begin{equation*}
S(n) := n \sum_{d_1\mid n}\sum_{d_2\mid n} \mu(d_1)\mu(d_2) s\left(\frac{n}{d_1},\frac{n}{d_2}\right).
\label{eq:def_Sn}
\end{equation*}

By symmetry with regards the variables $d_1$ and $d_2$, we have

\begin{equation*}
2S(n) = n \sum_{d_1\mid n}\sum_{d_2\mid n} \mu(d_1)\mu(d_2) \left( s\left(\frac{n}{d_1},\frac{n}{d_2}\right) + s\left(\frac{n}{d_2},\frac{n}{d_1}\right) \right).
\end{equation*}

Let $d$ be the greatest common divisor of $\frac{n}{d_1}$ and $\frac{n}{d_2}$, we have from proposition \ref{prop:sab_equals_sacbc}

\begin{equation*}
2S(n) = n \sum_{d_1 \mid n} \sum_{d_2 \mid n} \mu(d_1) \mu(d_2) \left( s\left(\frac{n}{dd_1},\frac{n}{dd_2}\right) + s\left(\frac{n}{dd_2},\frac{n}{dd_1}\right) \right).
\end{equation*}

Given that $\frac{n}{dd_1}$ and $\frac{n}{dd_2}$ are relatively prime, the Dedekind reciprocity law (see equation 4 page 3 in \cite{rademacher1972dedekind}) can be applied under the sum, and

\begin{equation*}
2S(n) = n \sum_{d_1 \mid n} \sum_{d_2 \mid n} \mu(d_1) \mu(d_2) \left( -\frac{1}{4} + \frac{1}{12}\left( \frac{d_2}{d_1} + \frac{d_1d_2}{n^2}d^2 + \frac{d_1}{d_2} \right) \right).
\end{equation*}

Now lets analyse the contribution of each term,

\begin{itemize}
	\item The term in $-\frac{1}{4}$ leads to a zero contribution since $n>1$ implies $\sum_{d \mid n}\mu(d)=0$, see Theorem 2.1 page 25 in \cite{apostol1976introduction}.
	\item The two terms in $\frac{d_1}{d_2}$ and in $\frac{d_2}{d_1}$ are equal by symetry, and are easily calculated as the product of two elementary sums (see for example \cite{apostol1976introduction})
	
\begin{equation*}
\sum_{d_1 \mid n} \sum_{d_2 \mid n} \mu(d_1) \mu(d_2) \frac{d_2}{d_1} = \left( \sum_{d_1 \mid n} \frac{\mu(d_1)}{d_1} \right)\left( \sum_{d_2 \mid n} \mu(d_2)d_2 \right) = \frac{\phi(n)}{n}(-1)^{\omega(m)}\phi(m).
\end{equation*}
	
	\item The last term in $\frac{d_1d_2}{n^2}d^2$ is calculated in proposition \ref{prop:formule_fn_Delange}
	
\begin{equation*}
\sum_{d_1 \mid n} \sum_{d_2 \mid n} \mu(d_1) \mu(d_2) \frac{d_1d_2}{n^2}d^2 = \frac{2^{\omega(n)}\phi(n)}{n}.
\end{equation*}
\end{itemize}

We easily get to

\begin{equation}\label{eq:formula_Sn}
S(n) = \frac{\phi(n)}{24}\left(2(-1)^{\omega(m)}\phi(m)+2^{\omega(n)}\right).
\end{equation}

Now considering equations \ref{eq:transform_jaj_thetanaa}, \ref{eq:main}, \ref{eq:formula_sum_ainUn_a2}, \ref{eq:expr_sum_ainUn_nunaa}, and \ref{eq:formula_Sn} we get after substitutions

\begin{equation*}
\sum_{j=1}^{\phi(n)}ja_j = \frac{\phi(n)}{n}\frac{\phi(n)}{6}\left( 2 n^2 + m (-1)^{\omega(m)} \right) + \frac{n\phi(n)}{4} - \frac{\phi(n)}{24}\left(2(-1)^{\omega(m)}\phi(m)+2^{\omega(n)}\right).
\end{equation*}

Considering that $\frac{\phi(n)}{n}=\frac{\phi(m)}{m}$\footnote{As a direct consequence of $\frac{\phi(n)}{n}=\prod_{p\in \mathbb{P}}\left(1-\frac{1}{p}\right)$, see Theorem 2.4 page 27 in \cite{apostol1976introduction}}, and that $\omega(n)=\omega(m)$, we see that the equation above simplifies into equation \ref{eq:Spence_formula}, which ends the proof.

\section{Conclusion}

The proof relies on three key ideas. The first idea is the reformulation of the problem via equation \ref{eq:transform_jaj_thetanaa} which enables to work on the set $\setU{n}$ and not to care about ordering the $a_j$ which is the job naturally done by the function $\theta_n$. The second key idea is to isolate the terms depending only on the fractional part function to be able ultimately to manipulate Dedekind sums. The third and main idea is to get the conditions to apply Dedekind's reciprocity law in order for the calculations to become possible via closed form formulaes. The result is not new, however the method may be applied to get formulaes for other sums of that kind, for which no closed form formulaes are known yet.

\section*{Acknowledgements} 

I would like to thank William Gasarch for his review of the paper and my wife Natallia for her continuous support.

\bibliographystyle{plain}
\bibliography{biblio}
\end{document}